\documentclass[11pt]{article}

\usepackage{amsfonts, amsmath,amssymb,array,latexsym, hyperref, amscd}

\usepackage{fancyhdr,a4wide}
\usepackage{amsthm}

\newtheorem{theorem}{Theorem}[section]
\newtheorem{lemma}[theorem]{Lemma}

\newtheorem{proposition}[theorem]{Proposition}

\newtheorem{definition}[theorem]{Definition}

\numberwithin{equation}{section}
\newcommand{\bdot}{{\tiny{\bullet}}}
\newcommand{\twiddle}[1]{\widetilde{#1}}
\newcommand{\gives}{\rightarrow}

\newcommand{\norm}[1]{\left|\left|#1\right|\right|}
\newcommand{\lr}[1]{\left(#1\right)}
\newcommand{\abs}[1]{\left|#1\right|}
\newcommand{\set}[1]{\left\{#1\right\}}
\newcommand{\E}[1]{\mathbb E\left[#1\right]}
\newcommand{\inprod}[2]{\left \langle #1,#2\right\rangle }

\def\Cov{\mathrm{Cov}} 
 
\newcommand{\x}{\times}
\newcommand{\R}{\mathbb R}

\newcommand{\Var}[1]{\mathrm{Var}\left[#1\right]}

\newcommand{\twomat}[4]{\ensuremath{ \left(\begin{array}{cc} #1 & #2 \\
#3 & #4 \end{array}\right)}}

\newcommand{\mN}{\mathcal N}
\newcommand{\z}[3]{z_{#1;#2}^{(#3)}}
\newcommand{\mF}{\mathcal F}
\newcommand{\mO}{\mathcal O}
\newcommand{\Prob}[1]{\mathbb P \left(#1\right)}

\begin{document}


\title{Random Neural Networks in the Infinite Width Limit as Gaussian Processes}

\author{Boris Hanin\\
Department of Operations Research and Financial Engineering\\
Princeton University}

\maketitle
\begin{abstract}
    This article gives a new proof that fully connected neural networks with random weights and biases converge to Gaussian processes in the regime where the input dimension, output dimension, and depth are kept fixed, while the hidden layer widths tend to infinity. Unlike prior work, convergence is shown assuming only moment conditions for the distribution of weights and for quite general non-linearities. 
\end{abstract}

\section{Introduction}
In the last decade or so neural networks, originally introduced in the 1940's and 50's \cite{hebb1949organization,rosenblatt1958perceptron}, have become indispensable tools for machine learning tasks ranging from computer vision \cite{krizhevsky2012imagenet} to natural language processing \cite{brown2020language} and reinforcement learning \cite{silver2017mastering}. Their empirical success has raised many new mathematical questions in approximation theory \cite{devore2020neural, yarotsky2017error,yarotsky2018optimal}, probability (see \S \ref{S:RMT} for some references), optimization/learning theory \cite{bartlett2020benign, belkin2019reconciling,jacot2018neural,zhang2021understanding} and so on. The present article concerns a fundamental probabilistic question about arguably the simplest networks, the so-called \textit{fully connected} neural networks,  defined as follows:
\begin{definition}[Fully Connected Network]\label{D:FC}
Fix a positive integer $L$ as well as $L+2$ positive integers $n_0,\ldots, n_{L+1}$ and a function $\sigma:\R\gives \R$. A fully connected depth $L$ neural network with input dimension $n_0$, output dimension $n_{L+1}$, hidden layer widths $n_1,\ldots, n_L$, and non-linearity $\sigma$ is any function  $x_\alpha\in \R^{n_0}\mapsto z_\alpha^{(L+1)}\in \R^{n_{L+1}}$ of the following form
\[
z_\alpha^{(\ell)} = \begin{cases}
W^{(1)}x_\alpha+b^{(1)},&\quad \ell=1\\
W^{(\ell)}\sigma(z_\alpha^{(\ell-1)})+b^{(\ell)},&\quad \ell=2,\ldots, L+1
\end{cases},
\]
where $W^{(\ell)}\in \R^{n_{\ell}\times n_{\ell-1}}$ are matrices, $b^{(\ell)}\in \R^{n_\ell}$ are vectors, and $\sigma$ applied to a vector is shorthand for $\sigma$ applied to each component. 
\end{definition}
The parameters $L,n_0,\ldots,n_{L+1}$ are called the \textit{network architecture}, and  $z_\alpha^{(\ell)}\in \R^{n_\ell}$ is called the \textit{vector of pre-activations at layer $\ell$} corresponding to input $x_\alpha.$ A fully connected network with a fixed architecture and given non-linearity $\sigma$ is therefore a finite (but typically high) dimensional family of functions, parameterized by the network weights (entries of the weight matrices $W^{(\ell)}$) and biases (components of bias vectors $b^{(\ell)}$). 

This article considers the mapping $x_\alpha\mapsto z_\alpha^{(L+1)}$ when the network's weights and biases are chosen independently at random and the hidden layer widths $n_1,\ldots, n_L$ are sent to infinity while the input dimension $n_0,$ output dimension $n_{L+1}$, and network depth $L$ are fixed. In this \textit{infinite width limit}, akin to the large matrix limit in random matrix theory (see \S \ref{S:why}), neural networks with random weights and biases converge to Gaussian processes (see \S \ref{S:disc} for a review of prior work). Unlike prior work Theorem \ref{T:NNGP}, our main result, is that this holds for general non-linearities $\sigma$ and distributions of network weights (cf \S \ref{S:main-result}). 

Moreover, in addition to establishing convergence of wide neural networks to a Gaussian process under weak hypotheses, the present article gives a mathematical take aimed at probabilists of some of the ideas developed in the recent monograph \cite{roberts2021principles}. This book, written in the language and style of theoretical physics by Roberts and Yaida, is based on research done jointly with the author. It represents a far-reaching development of the breakthrough work of Yaida \cite{yaida2020non}, which was the first to systematically explain how to compute \textit{finite width corrections} to infinite width Gaussian process limit of random neural networks for arbitrary depth, width, and non-linearity. Previously, such finite width (and large depth) corrections were only possible for some special observables in linear and ReLU networks \cite{ hanin2018neural,hanin2019finite,hanin2019products,hanin2021non, noci2021precise, zavatone2021exact}. The present article deals only with the asymptotic analysis of random neural networks as the width tends to infinity, leaving to future work a probabilistic elaboration of the some aspects of the approach to finite width corrections from \cite{roberts2021principles}. 

\subsection{Roadmap} The rest of this article is organized as follows. First, in \S \ref{S:why} we briefly motivate the study of neural networks with random weights. Then, in \S \ref{S:main-result} we formulate our main result, Theorem \ref{T:NNGP}. Before giving its proof in \S \ref{S:proof}, we first indicate in \S \ref{S:disc} the general idea of the proof and its relation to prior work.

\subsection{Why Random Neural Networks?}\label{S:why}
\subsubsection{Practical Motivations} It may seem at first glance that studying neural networks with random weights and biases is of no practical interest. After all, a neural network is only useful after it has been ``trained,'' i.e. one has found a setting of its parameters so that the resulting network function (at least approximately) interpolates a given training dataset of input-output pairs $(x,f(x))$ for an otherwise unknown function $f:\R^{n_0}\gives \R^{n_{L+1}}$. 

However, the vast majority of neural network training algorithms used in practice are variants of gradient descent starting from \textit{a random initialization} of the weight matrices $W^{(\ell)}$ and bias vectors $b^{(\ell)}$. Studying networks with random weights and biases therefore provides an understanding of the initial conditions for neural network optimization. 

Beyond illuminating the properties of networks at the start of training, the analysis of random neural networks can reveal a great deal about networks after training as well. Indeed, on a  heuristic level, just as the behavior of the level spacings of the eigenvalues of large random matrices is a surprisingly good match for emission spectra of heavy atoms \cite{wigner1958distribution}, it is not unreasonable to believe that certain coarse properties of the incredibly complex networks used in practice will be similar to those of networks with random weights and biases. More rigorously, neural networks used in practice often have many more tunable parameters (weights and biases) than the number of datapoints from the training dataset. Thus, at least in certain regimes, neural network training provably proceeds by an approximate linearization around initialization, since no one parameter needs to move much to fit the data. This so-called NTK analysis \cite{du2018gradient, fan2020spectra,huang2020dynamics,jacot2018neural,liu2020linearity} shows, with several important caveats related to network size and initialization scheme, that in some cases the statistical properties of neural networks at the start of training are the key determinants of their behavior throughout training.

\subsubsection{Motivation from Random Matrix Theory}\label{S:RMT} In addition to being of practical importance, random neural networks are also fascinating mathematical objects, giving rise to new problems in approximation theory \cite{daubechies2021nonlinear,devore2020neural,hanin2019universal,yarotsky2017error, yarotsky2018optimal}, random geometry \cite{hanin2019complexity,hanin2019deep}, and random matrix theory (RMT). Perhaps the most direct, though by no means only, connection to RMT questions is to set the network biases $b^{(\ell)}$ to zero and consider the very special case when $\sigma(t)=t$ is the identity (in the machine learning literature these are called deep linear networks). The network function
\begin{equation}\label{E:deep-linear}
z_\alpha^{(L+1)} = W^{(L+1)}\cdots W^{(1)}x_\alpha 
\end{equation}
is then a linear statistic for a product of $L+1$ independent random matrices. Such matrix models have been extensively studied, primarily in two regimes. The first is the multiplicative ergodic theorem regime \cite{crisanti2012products,furstenberg1963noncommuting,furstenberg1960products,ruelle1979ergodic}, in which all the layer widths $n_0,\ldots, n_{L+1}$ are typically set to a fixed value $n$ and the network depth $L$ tends to infinity. The second regime, where $L$ is fixed and the layer widths $n_\ell$ (i.e. matrix dimensions) tend to infinity, is the purview of free-probability \cite{nica2006lectures,voiculescu1986addition}.

In the presence of a  non-linearity $\sigma$, random neural network provide non-linear generalizations of the usual RMT questions. For instance, the questions taken up in this article are analogs of the joint normality of linear statistics of random matrix products in the free probability regime. Further, random neural networks give additional motivation for studying matrix products appearing in \eqref{E:deep-linear} when the matrix dimensions $n_\ell$ and the number of terms $L$ are simultaneously large. This  double scaling limit reveals new phenomena \cite{ahn2019fluctuations,akemann2012universal,akemann2014universal, akemann2019integrable,gorin2018gaussian,hanin2019products,hanin2021non} but is so far poorly understood relative to the ergodic or free regimes. 

Finally, beyond studying linear networks, random matrix theory questions naturally appear in neural network theory via non-linear analogs of the Marchenko-Pastur distribution for empirical covariance matrices of $z_\alpha^{(L+1)}$ when $\alpha\in A$ ranges over a random dataset of inputs \cite{adlam2019random,hastie2019surprises, peche2019note, pennington2019nonlinear} as well as through the spectrum of the input-output Jacobian \cite{hanin2019products, pennington2018emergence} and the NTK \cite{adlam2020neural,fan2020spectra}.

\subsection{Main Result}\label{S:main-result}
Our main result shows that under rather general conditions, when the weights $W^{(\ell)}$ and biases $b^{(\ell)}$ of a fully connected network are chosen at random, the resulting field $x_\alpha\mapsto z_\alpha^{(L+1)}$ converges to a centered Gaussian field with iid components when the input dimension $n_0$ and output dimension $n_{L+1}$ are held fixed but the hidden layer widths $n_1,\ldots, n_L$ tend to infinity. To give the precise statement in Theorem \ref{T:NNGP} below, fix a fully connected neural network with depth $L\geq 1$, input dimension $n_0$, output dimension $n_{L+1}$, hidden layer widths $n_1,\ldots, n_{L}\geq 1$, and non-linearity $\sigma:\R\gives \R$. We assume that $\sigma$ is absolutely continuous and that its almost-everywhere defined derivative (and hence $\sigma$ itself) is polynomially bounded:
\begin{equation}\label{E:sigma-prop}
\exists k>0\text{ s.t. }\forall x\in \R\quad \norm{\frac{\sigma'(x)}{1+\abs{x}^k}}_{L^\infty(\R)}< \infty.    
\end{equation}
All non-linearities used in practice satisfy these rather mild criteria. Further, let us write $W_{ij}^{(\ell)}$ for the entries of the weight matrices $W^{(\ell)}$ and $b_i^{(\ell)}$ for the components of the bias vectors $b^{(\ell)}$. For $\ell \geq 2$ the Definition \ref{D:FC} of fully connected networks means that the formula for the components of the pre-activations $z_\alpha^{(\ell)}$ at layer $\ell$ in terms of those for $z_\alpha^{(\ell-1)}$  reads
\begin{equation}\label{E:z-def}
\z{i}{\alpha}{\ell}:= b_i^{(\ell)}+\sum_{j=1}^{n_{\ell-1}}W_{ij}^{(\ell)}\sigma(\z{j}{\alpha}{\ell-1}),\qquad i=1,\ldots, n_{\ell},
\end{equation}
where we've denoted by  $\z{i}{\alpha}{\ell}$ the $i^{th}$ component of the $n_\ell$-dimensional vector of pre-activations $z_\alpha^{(\ell)}$ in layer $\ell$ corresponding to a network input $x_\alpha\in \R^{n_\ell}$. We make the following assumption on the network weights:
\begin{equation}\label{E:W-def}
    W_{ij}^{(\ell)}:=\lr{\frac{C_W}{n_{\ell-1}}}^{1/2}\widehat{W}_{ij}^{(\ell)},\qquad \widehat{W}_{ij}^{(\ell)}\sim \mu\quad \text{iid},
\end{equation}
where $\mu$ is a fixed probability distribution on $\R$ such that 
\begin{equation}\label{E:mu-W-def}
\mu\text{ has mean }0,\text{ variance }1\text{, and finite higher moments.}    
\end{equation}
We further assume the network biases are iid Gaussian\footnote{As explained in \S \ref{S:disc} the universality results in this article are simply not true if the biases are drawn iid from a fixed non-Gaussian distribution.} and independent of the weights:
\begin{equation}\label{E:b-def}
b_i^{(\ell)} \sim \mN(0,C_b)\quad \text{iid}.    
\end{equation}
In \eqref{E:W-def} and \eqref{E:b-def}, $C_W>0$ and $C_b\geq 0$ are fixed constants. These constants do not play an important role for the analysis in this article but will be crucial for followup work. With the network weights and biases chosen at random the vectors $z_\alpha^{(\ell)}$ are also random. Our main result is that, in the infinite width limit, they have independent Gaussian components. 
\begin{theorem}\label{T:NNGP}
Fix $n_0,n_{L+1}$ and a compact set $T\subseteq \R^{n_0}$. As the hidden layer widths $n_1,\ldots, n_L$ tend to infinity, the sequence of stochastic processes 
\[x_\alpha\in \R^{n_0}\quad \mapsto \quad  z_{\alpha}^{(L+1)}\in \R^{n_{L+1}}\]
converges weakly in $C^0(T,\R^{n_{L+1}})$ to a centered Gaussian process taking values in $\R^{n_{L+1}}$ with iid coordinates. The coordinate-wise covariance function  
\[
K_{\alpha\beta}^{(L+1)}:=\lim_{n_1,\ldots, n_L\gives\infty}\Cov\lr{\z{i}{\alpha}{L+1},\, \z{i}{\beta}{L+1}}
\]
for this limiting process satisfies the layerwise recursion 
\begin{equation}\label{E:K-rec}
    K_{\alpha\beta}^{(\ell+1)}=C_b+C_W\mathbb E\left[\sigma(z_\alpha)\sigma(z_\beta)\right],\qquad \lr{\begin{array}{c}
         z_\alpha  \\
         z_\beta 
    \end{array}}\sim \mN\lr{0,\twomat{K_{\alpha\alpha}^{(\ell)}}{K_{\alpha\beta}^{(\ell)}}{K_{\alpha\beta}^{(\ell)}}{K_{\beta\beta}^{(\ell)}}}
\end{equation}
for $\ell \geq 2$, with initial condition
\begin{equation}\label{E:K-initial}
    K_{\alpha\beta}^{(2)}=C_b+C_W\E{\sigma\lr{z_{1;\alpha}^{(1)}}\sigma\lr{z_{1;\beta}^{(1)}}},
\end{equation}
where the distribution of $(z_{1;\alpha}^{(1)}, z_{1;\beta}^{(1)})$ is determined via \eqref{E:z-def} by the distribution of weights and biases in the first layer and hence is not universal.
\end{theorem}
\noindent We prove Theorem \ref{T:NNGP} in \S \ref{S:proof}. First, we explain the main idea and review prior work.

\subsection{Theorem \ref{T:NNGP}: Discussion, Main Idea, and Relation to Prior Work}\label{S:disc}
At a high level, the proof of Theorem \ref{T:NNGP} (specifically the convergence of finite-dimensional distributions) proceeds as follows: 
\begin{enumerate}
    \item Conditional on the mapping $x_\alpha\mapsto z_\alpha^{(L)}$, the components of the neural network output $x_\alpha\mapsto z_\alpha^{(L+1)}$ are independent sums of $n_L$ independent random fields (see \eqref{E:z-def}), and hence, when $n_L$ is large, are each approximately Gaussian by the CLT. 
    \item The conditional covariance in the CLT from step $1$ is random at finite widths (it depends on $z_\alpha^{(L)}$). However, it has the special form of an average over $j=1,\ldots, n_L$ of the same function applied to each component $z_{j;\alpha}^{(L)}$ of the vector $z_\alpha^{(L)}$ of pre-activations at the last hidden layer. We call such objects \textit{collective observables} (see \S \ref{S:collective-proof} and \eqref{E:sigma-intro}). 
    \item While $z_{j;\alpha}^{(\ell)}$ are not independent at finite width when $\ell \geq 2$, they are weakly sufficiently correlated that a LLN still applies to any collective observable in the infinite width limit (see \S \ref{S:collective-proof}).
    \item The LLN from step 3 allows us to replace the random conditional covariance matrix from steps 1 and 2 by its expectation, asymptotically as $n_1,\ldots, n_L$ tend to infinity.
\end{enumerate}

We turn to giving a few more details on steps 1-4 and reviewing along the way the relation of the present article to prior work. The study of the infinite width limit for random neural networks dates back at least to  Neal \cite{neal1996priors}, who considered networks with one hidden layer:
\[
z_{i;\alpha}^{(2)} = b_i^{(2)} +\sum_{j=1}^{n_1}W_{ij}^{(2)}\sigma\lr{z_{j;\alpha}^{(1)}},\qquad z_{j;\alpha}^{(1)}=b_j^{(1)}+\sum_{k=1}^{n_0}W_{jk}^{(1)}x_{k;\alpha},
\]
where $i = 1,\ldots, n_2.$ In the shallow $L=1$ setting of Neal if in addition $n_2=1$, then neglecting the bias $b_1^{(2)}$ for the moment, the scalar field $z_{1;\alpha}^{(2)}$ is a sum of iid random fields with finite moments, and hence the asymptotic normality of its finite-dimensional distributions follows immediately from the multidimensional CLT. Modulo tightness, this explains why $z_{1;\alpha}^{(2)}$ ought to converge to a Gaussian field. Even this simple case, however, holds several useful lessons:
\begin{itemize} 
     \item If the distribution of the bias $b_1^{(2)}$ is fixed independent of $n_1$ is and non-Gaussian, then the distribution of $z_{1;\alpha}^{(2)}$ will not be Gaussian, even in the limit when $n_1\gives \infty$. 
    \item If the first layer biases $b_j^{(1)}$ are drawn iid from a fixed distribution $\mu_b$ and $\sigma$ is non-linear, then higher moments of $\mu_b$ will contribute to the variance of each neuron post-activation $\sigma(z_{j;\alpha}^{(1)})$, causing the covariance of the Gaussian field at infinite width to be non-universal. 
    \item Unlike in deeper layers, as long as $n_0$ is fixed, the distribution of each neuron pre-activation $z_{j;\alpha}^{(1)}$ in the first layer will not be Gaussian, unless the weights and biases in layer $1$ are themselves Gaussian. This explains why, in the initial condition \eqref{E:K-initial} the distribution is non-Gaussian in the first layer.
\end{itemize}
In light of the first two points, what should one assume about the bias distribution? There are, it seems, two options. The first is to assume that the variance of the biases tends to zero as $n_1\gives \infty$, putting them on par with the weights. The second, which we adopt in this article, is to declare all biases to be Gaussian. 

The first trick in proving Theorem \ref{T:NNGP} for general depth and width appears already when $L=1$ but the output dimension $n_2$ is at least two.\footnote{Neal \cite{neal1996priors} states erroneously on page 38 of his thesis that $z_{i;\alpha}^{(2)}$ and $z_{j;\alpha}^{(2)}$ will be independent because the weights going into them are independent. This is not true at finite width but becomes true in the infinite width limit.} In this case, even for a single network input $x_\alpha$, at finite values of the network width $n_1$ different components of the random $n_2$-dimensional vector $z_\alpha^{(2)}$ are not independent, due to their shared dependence on the vector $z_\alpha^{(1)}$. The key observation, which to the author's knowledge was first presented in \cite{lee2017deep}, is to note that the components of $z_\alpha^{(2)}$ are independent \textit{conditional on the first layer} (i.e. on $z_\alpha^{(1)}$) and are approximately Gaussian when $n_1$ is large by the CLT. The conditional variance, which captures the main dependence on $z_\alpha^{(1)}$, has the following form:
\begin{equation}\label{E:z-intro}
\Sigma_{\alpha\alpha}^{(2)} := \Var{z_{i;\alpha}^{(2)}~\big|~z_\alpha^{(1)}} = C_b + \frac{C_W}{n_1}\sum_{j=1}^{n_1} \sigma\lr{z_{j;\alpha}^{(1)}}^2. 
\end{equation}
This is an example of what we'll call a \textit{collective observable}, an average over all neurons in a layer of the same function applied to the pre-activations at each neuron (see \S \ref{S:collective-proof} for the precise definition). In the shallow $L=1$ setting, $\Sigma_{\alpha\alpha}^{(2)}$ is a sum of $n_1$ iid random variables with finite moments. Hence, by the LLN, it converges almost surely to its mean as $n_1\gives \infty$. This causes the components of $z_\alpha^{(2)}$ to become independent in the infinite width limit, since the source of their shared randomness, $\Sigma_{\alpha\alpha}^{(L+1)}$, can be replaced asymptotically by its expectation. 

The proof for general $L$ follows a similar pattern. Exactly as before, for any $0\leq \ell\leq L$, the components of the pre-activations at layer $\ell+1$ are still conditionally independent, given the pre-activations at layer $\ell$. When the width $n_\ell$ is large the conditional distribution of each component over any finite collection of inputs is therefore approximately Gaussian by the CLT. Moreover, the conditional covariance across network inputs has the form:
\begin{equation}\label{E:sigma-intro}
\Sigma_{\alpha\beta}^{(\ell+1)} := \Cov\lr{z_{i;\alpha}^{(\ell+1)},\, z_{i;\beta}^{(\ell+1)}~\big|~z_\alpha^{(\ell)},\, z_\beta^{(\ell)}} = C_b + \frac{C_W}{n_{\ell}}\sum_{j=1}^{n_{\ell}} \sigma\lr{z_{j;\alpha}^{(\ell)}}\sigma\lr{z_{j;\beta}^{(\ell)}}.
\end{equation}
The summands on the right hand side are no longer  independent at finite width if $\ell \geq 2$. However, $\Sigma_{\alpha\beta}^{(\ell+1)}$ are still collective observables, and the crucial point is to check that their dependence is sufficiently weak that we may still apply the LLN. Verifying this is the heart of the proof of Theorem \ref{T:NNGP} and is carried out in \S \ref{S:collective-proof}.

Let us mention that, in addition to the approach outlined above, other methods for showing that wide neural networks are asymptotically Gaussian processes are possible. In the prior article \cite{matthews2018gaussian}, for instance, the idea is to use that the entries of $z_\alpha^{(\ell)}$ are exchangeable and argue using an exchangeable CLT. This leads to some technical complications which, at least in the way the argument is carried out in \cite{matthews2018gaussian}, result in unnatural restrictions on the class of non-linearities and weight distributions considered there. Let us also mention that in the article \cite{lee2017deep}, the non-independence at finite width of the components of $z_{\alpha}^{(\ell)}$ for large $\ell$ was circumvented by considering only the sequential limit in which $n_\ell\gives \infty$ in order of increasing $\ell$. The effect is that for every $\ell$ the conditional covariance $\Sigma_{\alpha\beta}^{(\ell)}$ has already converged to its mean before $n_{\ell+1}$ is taken large. However, this way of taking the infinite width limit seems to the author somewhat unnatural and is any case not conducive to studying finite width corrections as in \cite{roberts2021principles,yaida2020non}, which we plan to take up in future work.

We conclude this section by pointing the reader to several other related strands of work. The first are articles such as \cite{daniely2016toward}, which quantify the magnitude of the difference
\[
\frac{1}{n_\ell}\sum_{i=1}^{n_\ell}z_{i;\alpha}^{(\ell)}z_{i;\beta}^{(\ell)}  - \lim_{n_1,\ldots, n_{\ell}\gives \infty}\E{\frac{1}{n_\ell}\sum_{i=1}^{n_\ell}z_{i;\alpha}^{(\ell)}z_{i;\beta}^{(\ell)}}
\]
between the empirical overlaps $n_{\ell}^{-1}\inprod{z_\alpha^{(\ell)}}{z_\beta^{(\ell)}}$ of  pre-activations and the corresponding infinite width limit uniformly over network inputs $x_\alpha,x_\beta$ in a compact subset of $\R^{n_0}$. In a similar vein are articles such as \cite{eldan2021non}, which give quantitative estimates at finite width for the distance from $x_\alpha\mapsto z_\alpha^{(\ell)}$ to a nearby Gaussian process. 

The second is the series of articles starting with the work of Yang \cite{yang2019scaling,yang2019tensori,yang2020tensorii,yang2020tensoriii}, which develops the study not only of initialization but also certain aspects of inference with infinitely wide networks using what Yang terms tensor programs. As part of that series, the article \cite{yang2019tensori} establishes that in the infinite width limit many different architectures become Gaussian processes. However, the arguments in those articles are significantly more technical than the ones presented here since they are focused on building the foundation for the tensor program framework. At any rate, to the best of the author's knowledge, no prior article addresses universality of the Gaussian process limit with respect to the weight distribution in deep networks (for shallow networks with $L=1$ this was considered by Neal in \cite{neal1996priors}). Finally, that random neural networks converge to Gaussian processes in the infinite width limit under various restrictions but for architectures other than fully connected is taken up in \cite{garriga2018deep,novak2018bayesian,yang2019tensori}.

\section{Proof of Theorem \ref{T:NNGP}}\label{S:proof}
Let us recall the notation. Namely, we fix a network depth $L\geq 1$, an input dimension $n_0\geq 1,$ an output dimension $n_{L+1}\geq 1$, hidden layer widths $n_1,\ldots, n_L\geq 1$ and a non-linearity $\sigma$ satisfying \eqref{E:sigma-def}. We further assume that the networks weights and biases are independent and random as in \eqref{E:W-def} and \eqref{E:b-def}. To prove Theorem \ref{T:NNGP} we must show that the random fields $x_\alpha\mapsto z_\alpha^{(L+1)}$ converge weakly in distribution to a Gaussian process in the limit where $n_1,\dots, n_L$ tend to infinity. We start with the convergence of finite-dimensional distributions. Let us therefore fix a collection
\[
x_A = \set{x_\alpha,\quad \alpha\in A}
\]
of $\abs{A}$ distinct network inputs in $\R^{n_0}$ and introduce for each $\ell=0,\ldots, L+1$, every $i=1,\ldots, n_\ell$, and all $\alpha\in A$ the vectorized notation
\[
z_{i;A}^{(\ell)} := \lr{z_{i;\alpha}^{(\ell)},\, \alpha\in A}\in \R^{\abs{A}},\qquad z_{A}^{(\ell)}:= \lr{z_{i;A}^{(\ell)},\, i=1,\ldots, n_\ell}\in \R^{n_\ell \times \abs{A}}.
\]
The following result states that the distribution of the random variable $z_A^{(L+1)}$ with values in $\R^{n_{L+1}\times \abs{A}}$ converges to that of the claimed Gaussian field.
\begin{proposition}[Convergence of Finite-Dimensional Distributions]\label{P:fdd}
Fix $L\geq 1$ and $n_0,n_{L+1}.$ The distribution of $z_A^{(L+1)}$ converges weakly as $n_1,\ldots, n_L\gives \infty$ to that of a centered Gaussian in $\R^{n_{L+1}\times \abs{A}}$ with iid rows for which the covariance 
\[
K_{\alpha\beta}^{(L+1)}=\lim_{n_1,\ldots, n_L\gives\infty} \Cov\lr{z_{i;\alpha}^{(L+1)},z_{i;\beta}^{(L+1)}} ,\qquad \alpha,\beta\in A
\]
between the entries in each row satisfies the recursion \eqref{E:K-rec} with initial condition \eqref{E:K-initial}.
\end{proposition}
Once we have proved Proposition \ref{P:fdd} in \S \ref{S:fdd-proof}, it remains to show tightness. For this, we fix a compact subset $T\subseteq \R^{n_0}$. The tightness of $x_\alpha\mapsto z_\alpha^{(L+1)}$ in $C^0(T,\R^{n_{L+1}})$ follows immediately from the Arzel\`a-Ascoli Theorem and the following result, which we prove in \S \ref{S:tightness-proof}.
\begin{proposition}
[High Probability Equicontinuity and Equiboundedness of $z_\alpha^{(L+1)}$]\label{P:tightness}
For every $L\geq 1,\,\epsilon>0$ there exists $C=C(\epsilon,\sigma,T,L,C_b,C_W)>0$ so that 
\begin{equation}\label{E:aa-hyp}
\sup_{x_\alpha,x_\beta\in T}\frac{ \norm{z_\alpha^{(L+1)}-z_\beta^{(L+1)}}_2}{\norm{x_\alpha-x_\beta}_2}\leq C \qquad \text{and}\qquad \sup_{x_\alpha\in T}\norm{z_\alpha^{(L+1)}}\leq C
\end{equation}
with probability at least $1-\epsilon$.
\end{proposition}

\subsection{Finite-Dimensional Distributions: Proof of Proposition \ref{P:fdd}}\label{S:fdd-proof}
We will prove Proposition \ref{P:fdd} in two steps. First, we prove a special case in which we keep the weights in layer $1$ completely general as in \eqref{E:W-def} but take weights in layers $\ell\geq 2$ to be independent Gaussians:
\[
W_{ij}^{(\ell)}\sim \mN\lr{0,C_Wn_{\ell-1}^{-1}},\qquad \text{iid}.
\]
We continue to assume (as in the statement of Theorem \ref{T:NNGP}) that all biases are Gaussian:
\[
b_i^{(\ell)}\sim \mN(0,C_b),\qquad \text{iid}.
\]
The argument in this case is the technical heart of this paper and is presented in \S \ref{S:Gaussian-proof} - \S\ref{S:collective-proof}. Ultimately, it relies on the analysis of collective observables, which we isolate in \S \ref{S:collective-proof}. A simple Lindeberg swapping argument and induction on layer detailed in \S \ref{S:gen-weights} allows us to extend Proposition \ref{P:fdd} to general weights in layers $\ell\geq 2$ from the Gaussian case.  

\subsubsection{Proof of Proposition \ref{P:fdd} with Gaussian Weights in Layers $\ell\geq 2$}\label{S:Gaussian-proof} Fix 
\[
\Xi = \lr{\xi_i,\,i=1,\ldots, n_{L+1}}\in\R^{n_{L+1}\times \abs{A}},\qquad \xi_i=\lr{\xi_{i;\alpha},\,i=1,\ldots, n_{L+1},\, \alpha\in A}\in \R^{\abs{A}}
\]
and consider the characteristic function
\[
\chi_A(\Xi) = \E{\exp\left[-i\inprod{z_A^{(L+1)}}{\Xi}\right]}=\E{\exp\left[-i\sum_{\alpha\in A}\sum_{i=1}^{n_{L+1}}z_{i;\alpha}^{(L+1)}\xi_{i;\alpha}\right]}
\]
of the random variable $z_A^{(L+1)}\in \R^{n_{L+1}\times \abs{A}}$. By Levy's continuity theorem, it is sufficient to show that 
\begin{equation}\label{E:levy-goal}
    \lim_{n_1,\ldots,n_L\gives \infty}\chi_A(\Xi)  = \exp\left[-\frac{1}{2}\sum_{i=1}^{n_{L+1}} \inprod{K_A^{(L+1)}\xi_i}{\xi_i}\right],
\end{equation}
where
\[
K_A^{(L+1)} = \lr{K_{\alpha\beta}^{(L+1)}}_{\alpha,\beta\in A}
\]
is the matrix defined by the recursion \eqref{E:K-rec} with initial condition \eqref{E:K-initial}. Writing 
\begin{equation}\label{E:F-def}
\mF_\ell : = \text{filtration defined by }\set{W^{(\ell')}, b^{(\ell')},\, \, \ell'=1,\ldots, \ell},    
\end{equation}
we may use the tower property to write
\begin{align}
      \label{E:cond-chi}\chi_A(\Xi) &= \E{\E{\exp\left[-i\inprod{z_A^{(L+1)}}{\Xi}\right]~\big|~\mF_{L}}}.
\end{align}
Note that conditional on $\mF_L$, the random vectors $z_{i;A}^{(L+1)}\in \R^{\abs{A}}$ in layer $L+1$ for each $i=1,\ldots, n_{L+1}$ are iid Gaussians, since we've assumed for now that weights in layers $\ell\geq 2$ are Gaussian. Specifically, 
\[
z_{i;A}^{(L+1)} \stackrel{d}{=} \lr{\Sigma_A^{(L+1)}}^{1/2}G_i,\qquad G_i\sim \mN\lr{0, \mathrm{I}_{\abs{A}}} \text{  iid}\qquad 1\leq i \leq n_{L+1},
\]
where for any $\alpha,\beta\in A$ the conditional covariance is
\begin{equation}\label{E:sigma-def}
    \lr{\Sigma_A^{(L+1)}}_{\alpha\beta} = \Cov\lr{z_{i;\alpha}^{(L+1)},\, z_{i;\beta}^{(L+1)}~\big|~\mF_{L}}=C_b + \frac{C_W}{n_L}\sum_{j=1}^{n_L}\sigma\lr{z_{j;\alpha}^{(L)}}\sigma\lr{z_{j;\beta}^{(L)}}.
\end{equation}
Using \eqref{E:cond-chi} and the explicit form of the characteristic function of a Gaussian reveals
\begin{align}
      \label{E:chi-form}\chi_A(\Xi) &= \E{\exp\left[-\frac{1}{2}\sum_{i=1}^{n_{L+1}} \inprod{\Sigma_A^{(L+1)}\xi_i}{\xi_i} \right] }.
\end{align}
The crucial observation is that each entry of the conditional covariance matrix $\Sigma_A^{(L+1)}$ is an average over $j=1,\ldots, n_L$ of the same fixed function applied to the vector $z_{j;A}^{(L)}$. While $z_{j;A}^{(L)}$ are not independent at finite values of $n_1,\ldots, n_{L-1}$ for $L>1$, they are sufficiently weakly correlated that a weak law of large numbers still holds: 
\begin{lemma}\label{L:collective-sigma}
Fix $n_0,n_{L+1}$. There exists a $\abs{A}\x \abs{A}$ PSD matrix 
\[
K_A^{(L+1)} = \lr{K_{\alpha\beta}^{(L+1)}}_{\alpha,\beta\in A}
\]
such that for all $\alpha,\beta\in A$
\[
\lim_{n_1,\ldots, n_L\gives \infty}\E{\lr{\Sigma_A^{(L+1)}}_{\alpha\beta}}= K_{\alpha\beta}^{(L+1)}\qquad \text{and}\qquad   \lim_{n_1,\ldots, n_L\gives \infty}\Var{\lr{\Sigma_A^{(L+1)}}_{\alpha\beta}}= 0.
\]
\end{lemma}
\begin{proof}
Lemma \ref{L:collective-sigma} is a special case of Lemma \ref{L:collective-properties} (see \S \ref{S:collective-proof}).
\end{proof}
\vspace{.3cm}
Lemma \ref{L:collective-sigma} implies that $\Sigma_A^{(L+1)}$ converges in distribution to $K_A^{(L+1)}$. In view of \eqref{E:chi-form} and the definition of weak convergence this immediately implies \eqref{E:levy-goal}. It therefore remains to check that $K_A^{(L+1)}$ satisfies the desired recursion. For this, note that at any values of $n_1,\ldots, n_L$ we find
\begin{align*}
    \Cov\lr{z_{i;\alpha}^{(L+1)},\, z_{i;\beta}^{(L+1)}} &= \E{\lr{\Sigma_A^{(L+1)}}_{\alpha\beta}}
    = C_b + C_W\E{\sigma\lr{z_{1;\alpha}^{(L)}}\sigma\lr{z_{1;\beta}^{(L)}}}.
\end{align*}
When $L=1$, we therefore see that
\[
\Cov\lr{z_{i;\alpha}^{(2)}, z_{i;\beta}^{(2)}} = C_b +C_W\E{\sigma(z_{i;\alpha}^{(1)})\sigma(z_{i;\beta}^{(1)})},
\]
where the law of $(z_{i;\alpha}^{(1)},z_{i;\beta}^{(1)})$ is determined by the distribution $\mu_W$ of weights in layer $1$ and does not depend on $n_1$. This confirms the initial condition \eqref{E:K-initial}. Otherwise, if $L>1$, the convergence of finite-dimensional distributions that we've already established yields 
\begin{align*}
    K_{\alpha\beta}^{(L+1)} =& \lim_{n_1,\ldots, n_L\gives \infty} \Cov\lr{z_{i;\alpha}^{(L+1)},\, z_{i;\beta}^{(L+1)}} =\lim_{n_1,\ldots, n_{L-1}\gives \infty} \lr{C_b + C_W\E{\sigma\lr{z_{1;\alpha}^{(\ell)}}\sigma\lr{z_{1;\beta}^{(\ell)}}}}.
\end{align*}
Since $\sigma$ is continuous we may invoke the continuous mapping theorem to conclude that 
\begin{align*}
    K_{\alpha\beta}^{(L+1)}  &=C_b + C_W\mathbb E_{(z_\alpha,z_\beta)\sim G(0,K^{(L)})}\left[\sigma(z_\alpha)\sigma(z_\beta)\right],
\end{align*}
which confirms the recursion \eqref{E:K-rec}. This completes the proof that the finite-dimensional distributions of $z_\alpha^{(L+1)}$ converge to those of the desired Gaussian process, modulo two issues. First, we must prove Lemma \ref{L:collective-sigma}. This is done in \S \ref{S:collective-proof} by proving a more general result, Lemma \ref{L:collective-properties}. Second, we must remove the assumption that the weights in layers $\ell\geq 2$ are Gaussian. This is done in \S \ref{S:gen-weights}. \hfill $\square$

\subsubsection{Collective Observables with Gaussian Weights: Generalizing Lemma \ref{L:collective-sigma}}\label{S:collective-proof}
This section contains the key technical argument in our proof of Proposition \ref{P:fdd}. To state the main result, define a \textit{collective observable at layer $\ell$} to be any random variable of the form
\[
\mO_{n_\ell,f;A}^{(\ell)}:= \frac{1}{n_\ell} \sum_{i=1}^{n_\ell} f(z_{i;A}^{(\ell)}),
\]
where $f:\R^{\abs{A}}\gives \R$ is measurable and polynomially bounded:
    \[
    \exists C>0,\, k\geq 1\text{ s.t. }\forall z\in \R^{\abs{A}}\qquad \abs{f(z)}\leq C\lr{1+\norm{z}_2^k}.
    \]
We continue to assume, as \S \ref{S:Gaussian-proof}, that the weights (and biases) in layers $\ell\geq 2$ are Gaussian and will remove this assumption in \S \ref{S:gen-weights}. Several key properties of collective observables are summarized in the following Lemma: 
\begin{lemma}[Key Properties of Collective Observables]\label{L:collective-properties}
Let $\mO_{n_\ell,f;A}^{(\ell)}$ be a collective observable at layer $\ell$. There exists a deterministic scalar $\overline{O}_{f;A}^{(\ell)}$ such that 
\begin{equation}\label{E:collective-mean}
\lim_{n_1,\ldots, n_{\ell-1}\gives \infty}\E{\mO_{n_\ell, f;A}^{(\ell)}}=\overline{O}_{f;A}^{(\ell)}.
\end{equation}
Moreover, 
\begin{equation}\label{E:self-avg}
\lim_{n_1,\ldots, n_\ell\gives \infty}\Var{\mO_{n_\ell,f;A}^{(\ell)}}=0.
\end{equation}
Hence, we have the following convergence in probability
\[
\lim_{n_1,\ldots, n_\ell\gives \infty} \mO_{n_\ell,f;A}^{(\ell)} \quad \stackrel{p}{\longrightarrow}\quad \overline{O}_{f;A}^{(\ell)}.
\]
\end{lemma}
\begin{proof}
This proof is by induction on $\ell$, starting with $\ell=1$. In this case, $z_{i;A}^{(1)}$ are independent for different $i$. Moreover, for each $i,\alpha$ we have
\[
z_{i;\alpha}^{(1)}= b_i^{(1)} + \sum_{j=1}^{n_0}W_{ij}^{(1)} x_{j;\alpha} =b_i^{(1)} + \lr{\frac{C_W}{n_0}}^{1/2}\sum_{j=1}^{n_0}\widehat{W}_{ij}^{(1)} x_{j;\alpha}.
\]
Hence, $z_{i;\alpha}^{(1)}$ have finite moments since $b_i^{(1)}$ are iid Gaussian and $\widehat{W}_{ij}^{(1)}$ are mean $0$ with finite higher moments. In particular, since $f$ is polynomially bounded, we find for every $n_1$ that 
\[
\E{\mO_{n_1,f;A}^{(1)}} = \E{f(z_{1;A}^{(1)})},
\]
which is finite and independent of $n_1$, confirming \eqref{E:collective-mean}. Further, $\mO_{n_1,f;A}^{(1)}$ is the average of $n_1$ iid random variables with all moments finite. Hence, \eqref{E:self-avg} follows by the weak law of large numbers, completing the proof of the base case. 

Let us now assume that we have shown \eqref{E:collective-mean} and \eqref{E:self-avg} for all $\ell=1,\ldots, L$. We begin by checking that  \eqref{E:collective-mean} holds at layer $L+1$. We have
\begin{equation}\label{E:mean-red}
\E{\mO_{n_{L+1}, f;A}^{(L+1)}} = \E{f(z_{1;A}^{(L+1)})}.    
\end{equation}
Since the weights and biases in layer $L+1$ are Gaussian and independent of $\mF_L$, we find
\begin{equation}\label{E:z-Gaussian}
z_{1;A}^{(L+1)} \stackrel{d}{=} \lr{\Sigma_A^{(L+1)}}^{1/2} G ,
\end{equation}
where $\Sigma_A^{(L+1)}$ is the conditional covariance defined in \eqref{E:sigma-def} and $G$ is an $\abs{A}$-dimensional standard Gaussian. The key point is that $\Sigma_A^{(L+1)}$ is a collective observable at layer $L$. Hence, by the inductive hypothesis, there exists a PSD matrix $\overline{\Sigma}_A^{(L+1)}$ such that $\Sigma_A^{(L+1)}$ converges in probability to $\overline{\Sigma}_A^{(L+1)}$ as $n_1,\ldots, n_L\gives \infty$. To establish \eqref{E:collective-mean} it therefore suffices in view of \eqref{E:mean-red} to check that
\begin{equation}\label{E:mean-goal}
\lim_{n_1,\ldots, n_{L}\gives \infty} \E{f\lr{\lr{\Sigma_A^{(L+1)}}^{1/2} G}} = \E{f\lr{\lr{\overline{\Sigma}_A^{(L+1)}}^{1/2} G}},
\end{equation}
where the right hand side is finite since $f$ is polynomially bounded and all polynomial moments of $G$ are finite. To establish \eqref{E:mean-goal}, let us invoke the Skorohod representation theorem to find a common probability space on which there are versions of $\Sigma_A^{(L+1)}$ -- which by an abuse of notation we will still denote by $\Sigma_A^{(L+1)}$ -- that converge to $\overline{\Sigma}_A^{(L+1)}$ almost surely. Next, note that since $f$ is polynomially bounded we may repeatedly apply $ab\leq \frac{1}{2}(a^2+b^2)$ to find
\begin{equation}\label{E:poly-f}
\abs{f((\Sigma_A^{(L+1)})^{1/2} G)}\leq p\lr{(\Sigma_A^{(L+1)})^{1/2}} + q(G),
\end{equation}
where $p$ is a polynomial in the entries of $(\Sigma_A^{(L+1)})^{1/2}$, $q$ a polynomial in the entries of $G$, and the polynomials $p,q$ don't depend on $n_1,\ldots, n_{L}$. The continuous mapping theorem shows that 
\[
\lim_{n_1,\ldots, n_L\gives \infty}\E{p\lr{(\Sigma_A^{(L+1)})^{1/2}}} = \E{p\lr{(\overline{\Sigma}_A^{(L+1)})^{1/2}}}.
\]
Thus, since all moments of Gaussian are finite, \eqref{E:mean-goal} follows from the generalized dominated convergence theorem. It remains to argue that \eqref{E:self-avg} holds at layer $L+1$. To do this, we write
\begin{align}
    \label{E:var-red}\Var{\mO_{n_{L+1},f;A}^{(L+1)}} = \frac{1}{n_{L+1}} \Var{f(z_{1;A}^{(L+1)})} + \lr{1-\frac{1}{n_{L+1}}}\Cov\lr{f(z_{1;A}^{(L+1)}),\, f(z_{2;A}^{(L+1)})}.
\end{align}
Observe that  
\[
\Var{f(z_{1;A}^{(L+1)})} \leq \E{f(z_{1;A}^{(L+1)})^2} = \E{\frac{1}{n_{L+1}}\sum_{i=1}^{n_{L+1}}\lr{f(z_{i;A}^{(L+1)})}^2}<\infty,
\]
since we already showed that \eqref{E:collective-mean} holds at layer $L+1$. Next, recall that, conditional on $\mF_L$, neurons in layer $L+1$ are independent. The law of total variance and Cauchy-Schwartz yield
\begin{align}
 \notag \abs{\Cov\lr{f(z_{1;A}^{(L+1)}),\, f(z_{2;A}^{(L+1)})}}&=\abs{\Cov\lr{\E{f(z_{1;A}^{(L+1)})~|~\mF_{L}},\, \E{f(z_{2;A}^{(L+1)})~|~\mF_L}}}\\
 \label{E:var-reduce}&\leq \Var{\E{f(z_{1;A}^{(L+1)})~|~\mF_{L}}}.
\end{align}
Using \eqref{E:z-Gaussian} and the polynomial estimates \eqref{E:poly-f} on $f$, we conclude that the conditional expectation on the previous line  is some polynomially bounded function of the components of $(\Sigma_{A}^{(L+1)})^{1/2}$. Hence, we may apply dominated convergence as above to find 
\[
\lim_{n_1,\ldots,n_L\gives \infty} \Var{\E{f(z_{1;A}^{(L+1)})~|~\mF_{L}}} = \Var{\E{f(\overline{\Sigma}_A^{1/2})G}} = 0,\qquad G\sim \mN(0,\mathrm{I}_{\abs{A}})
\]
since $\E{f(\overline{\Sigma}_A^{1/2})G}$ is a constant. This proves \eqref{E:self-avg} for observables at layer $L+1$ and completes the proof of Lemma \ref{L:collective-properties}.
\end{proof}

\subsubsection{Proof of Proposition \ref{P:fdd} for General Weights}\label{S:gen-weights} In this section, we complete the proof of Proposition \ref{P:fdd} by removing the assumption from \S \ref{S:Gaussian-proof} that weights in layers $\ell\geq 2$ are Gaussian. To do this, let us introduce some notation. Let us write
\[
x_\alpha\mapsto z_\alpha^{(\ell)}
\]
for the pre-activations at layers $\ell=0,\ldots, L+1$ of a random fully connected network in which, as in the general case of Theorem \ref{T:NNGP}, all weights and biases are independent, biases are Gaussian as in \eqref{E:b-def} and weights in all layers are drawn from a general centered distribution with the correct variance and finite higher moments as in \eqref{E:W-def} and \eqref{E:mu-W-def}. Next, let us write
\[
x_\alpha\mapsto \twiddle{z}_\alpha^{(\ell)},
\]
for the vector of pre-activations obtained by replacing, in each layer $\ell=2,\ldots, L+1$, all weights $W_{ij}^{(\ell)}$ by iid centered Gaussians with variance $C_W/n_{\ell-1}$. We already saw that the distribution of $\twiddle{z}_A^{{(L+1)}}$ converges weakly to that of the desired Gaussian in the infinite width limit. Our goal is thus to show that, as $n_1,\ldots, n_L$ tend to infinity, $z_A^{{(L+1)}}$ and $\twiddle{z}_A^{{(L+1)}}$ converge weakly to the same distribution. We will proceed by induction on $L$. When $L=0$ the claim is trivial since, by construction, the weight and bias distributions in layer $1$ are identical in both $z_\alpha^{(1)}$ and $\twiddle{z}_\alpha^{(1)}$ (recall that when we proved Proposition \ref{P:fdd} in \S \ref{S:Gaussian-proof} we had Gaussian weights only starting from layer $2$ and general weights in layer $1$.)

Suppose therefore that we have shown the claim for $\ell=0,\ldots, L$. By the Portmanteau theorem and the density of smooth compactly supported functions in the space of continuous compactly supported functions equipped with the supremum norm, it suffices to show that for any smooth function $g:\R^{n_{L+1}\times\abs{A}}\gives \R$ with compact support we have
\begin{equation}\label{E:g-goal}
\lim_{n_1,\ldots, n_L\gives \infty}\lr{ \E{g(z_A^{(L+1)})} - \E{g(\twiddle{z}_A^{(L+1)})} }= 0.    
\end{equation}
To check \eqref{E:g-goal}, let us define an intermediate object:
\[
z_\alpha^{(L+1),\bdot} = b^{(L+1)}+W^{(L+1),\bdot}\sigma\lr{z_\alpha^{(L)}},
\]
where the entries $W_{ij}^{(L+1),\bdot}$ of $W^{(L+1),\bdot}$ are iid Gaussian with mean $0$ and variance $C_W/n_L$. That is, we take the vector $\sigma(z_\alpha^{(L)})$ of post-activations from layer $L$ obtained by using general weights in layers $1,\ldots, L$ and use Gaussian weights only in layer $L+1$. Our first step in checking \eqref{E:g-goal} is to show that it this relation holds when $z_A^{(L+1)}$ is replaced by $z_A^{(L+1),\bdot}$. 
\begin{lemma}\label{L:Lindeberg}
Let $x_A=\set{x_\alpha,\,\in A}$ be a finite subset of $\R^{n_0}$ consists of $\abs{A}$ distinct elements. Fix in addition $n_{L+1}\geq 1$ and a smooth compactly supported function $g:\R^{n_{L+1}\x\abs{A}}\gives \R$. There exists $C>0$ and a collective observable $\mO_{n_L,f;A}^{(L)}$ at layer $L$ so that 
\[
\abs{\E{g\lr{z_A^{(L+1)}}}-\E{g\lr{z_A^{(L+1),\bdot}}}}\leq C n_{L+1}^3 n_{L}^{-1/2} \E{\mO_{n_L,f;A}^{(L)}}.
\]
\end{lemma}
\begin{proof}
This is a standard Lindeberg swapping argument. Namely, for each $\alpha\in A$ and  $k=0,\ldots, n_L$ define
\[
z_\alpha^{(L+1),k}:=b^{(L+1)}+W^{(L+1),k}\sigma\lr{z_\alpha^{(L)}},
\]
where the first $k$ entries of each row of $W^{(L+1),k}$ are iid Gaussian with mean $0$ and variance $C_W/n_L$, while the remaining entries are $(C_W/n_L)^{-1/2}$ times iid draws $\widehat{W}_{ij}^{(L+1)}$ from the general distribution $\mu$ of network weights, as in \eqref{E:W-def} and \eqref{E:mu-W-def}. With this notation, we have
\[
z_\alpha^{(L+1)}=z_\alpha^{(L),0},\qquad \twiddle{z}_\alpha^{(L+1),\bdot} = z_\alpha^{(L),n_L}.
\]
Thus, 
\[
\E{g\lr{z_A^{(L+1)}}}-\E{g\lr{\twiddle{z}_A^{(L+1)}}} = \sum_{k=1}^{n_L} \E{g\lr{z_A^{(L+1),k-1}}}-\E{g\lr{z_A^{(L+1),k}}}.
\]
For any $Z\in \R^{n_{\ell+1}\times \abs{A}}$ and
\[
\delta Z =\lr{\delta z_{i;\alpha}\, i = 1,\ldots, n_{\ell+1},\, \alpha\in A}\in \R^{n_{\ell+1}\times \abs{A}}
\]
consider the third order Taylor expansion of $g$ around $Z$.
\begin{align*}
g(Z+\delta Z) &= g(Z) + \sum_{\substack{\alpha\in A\\ i=1,\ldots, n_{L+1}}} g_{i;\alpha} \delta z_{i;\alpha} +  \sum_{\substack{\alpha_1,\alpha_2\in A\\ i_1,i_2=1,\ldots, n_{L+1}}} g_{i_1,i_2;\alpha_1,\alpha_2} \delta z_{i_1;\alpha_1}\delta z_{i_2;\alpha_2}\\
&+ O\lr{ \sum_{\substack{\alpha_1,\alpha_2,\alpha_3\in A\\ i_1,i_2,i_3=1,\ldots, n_{L+1}}} \abs{\delta z_{i_1;\alpha_1}\delta z_{i_2;\alpha_2}\delta z_{i_3;\alpha_3}}}.
\end{align*}
Let us write
\[
z_{i;\alpha}^{(L+1),k-1} = z_{i;\alpha}^{(L+1),k} + n_L^{-1/2} Y_{i,k;\alpha},\qquad \delta Z_{i,k;\alpha} = C_W^{1/2} \lr{\twiddle{W}_{ik}^{(L+1)}-\widehat{W}_{ik}^{(L+1)}}\sigma(z_{k;\alpha}^{(L)}),
\]
where $\twiddle{W}_{ik}^{(L+1)}\sim \mN(0,1)$. Then, Taylor expanding $g$ to third order around $Z_k=z_{i;\alpha}^{(L+1),k}$ and, using that the first two moments of $(C_W n_L^{-1})^{1/2}\widehat{W}_{ij}^{(L)}$ match those of $\mN(0,C_Wn_L^{-1})$, we find that
\[
\E{g\lr{z_A^{(L+1),k-1}}}-\E{g\lr{\twiddle{z}_A^{(L+1),k}}} = O\lr{n_L^{-3/2} n_{L+1}^3  \E{p\lr{\abs{\sigma(z_{k;\alpha}^{(L)})},\,\alpha\in A)}} },
\]
where $p$ is a degree $3$ polynomial in the $\abs{A}$-dimensional vector of absolute values $|\sigma(z_{k;\alpha}^{(\ell)})|,\,\alpha\in A.$ Summing this over $k$ completes the proof. 
\end{proof}

\noindent To make use of Lemma \ref{L:Lindeberg} let us consider any collective observable $\mO_{n_L,f;A}^{_{(L)}}$ at layer $L$. Recall that by \eqref{E:mean-red} and \eqref{E:var-red} both the mean and variance of $\mO_{n_L,f;A}^{_{(L)}}$ depend only on the distributions of finitely many components of the vector $z_A^{(L)}$. By the inductive hypothesis we therefore find 
\begin{equation}\label{E:mean-equiv}
\lim_{n_1,\ldots, n_L\gives \infty} \E{\mO_{n_L,f;A}^{(L)}} = \lim_{n_1,\ldots, n_L\gives \infty} \E{\twiddle{\mO}_{n_L,f;A}^{(L)}},
\end{equation}
where the right hand side means that we consider the same collective observable but for $\twiddle{z}_A^{(L)}$ instead of $z_A^{(L)}$, which exists by Lemma \ref{L:collective-properties}. Similarly, again using Lemma \ref{L:collective-properties}, we have
\begin{equation}\label{E:var-equiv}
\lim_{n_1,\ldots, n_L\gives \infty} \Var{\mO_{n_L,f;A}^{(L)}} = 0.  
\end{equation}
Therefore, we conclude that 
\begin{equation}\label{E:dist-conv}
\mO_{n_L,f;A}^{(L)}- \twiddle{\mO}_{n_L,f;A}^{(L)} \quad \stackrel{d}{\longrightarrow}\quad 0,\qquad \text{as }n_1,\ldots, n_L\gives \infty.    
\end{equation}
Note that by \eqref{E:mean-equiv} the mean of any collective observable $\E{\mO_{n_L,f;A}^{(L)}}$ is bounded independent of $n_1,\ldots, n_L$ since we saw in Lemma \ref{L:collective-properties} that the limit exists and is bounded when using Gaussian weights. Since $n_{L+1}$ is fixed and finite, the error term $n_L^{-1/2}n_{L+1}^3 \E{\mO_{n_L,f;A}^{_{(L)}}}$ in Lemma \ref{L:Lindeberg} is therefore tends to zero as $n_1,\ldots, n_L\gives \infty$, and \eqref{E:g-goal} is reduced to showing that 
\begin{equation}\label{E:g-goal-2}
\lim_{n_1,\ldots, n_L\gives \infty}\lr{ \E{g(z_A^{(L+1),\bdot})} - \E{g(\twiddle{z}_A^{(L+1)})} }= 0.    
\end{equation}
This follows from \eqref{E:var-equiv} and the inductive hypothesis. Indeed, by construction, conditional on the filtration $\mF_L$ defined by weights and biases in layers up to $L$ (see \eqref{E:F-def}), the $\abs{A}$-dimensional vectors $z_{i;A}^{(L+1),\bdot}$ are iid Gaussians:
\[
z_{i;A}^{(L+1),\bdot} \stackrel{d}{=} \lr{\Sigma_A^{(L+1)}}^{1/2}G_i,\qquad G_i\sim \mN(0, \mathrm{I}_{\abs{A})}\quad iid,
\]
where $\Sigma_A^{(L+1)}$ is the conditional covariance matrix from \eqref{E:sigma-def}. The key point, as in the proof with all Gaussian weights, is that each entry of the matrix $\Sigma_A^{(L+1)}$ is a collective observable at layer $L$. Moreover, since the weights and biases in the final layer are  Gaussian for $z_A^{(L+1),\bdot}$ the conditional distribution of $g(z_A^{(L+1),\bdot})$ given $\mF_L$ is completely determined by $\Sigma_A^{(L+1)}$. In particular, since $g$ is bounded and continuous, we find that 
\[
\E{g(z_A^{(L+1),\bdot})} - \E{g(\twiddle{z}_A^{(L+1)})} = \E{h\lr{\Sigma_A^{(L+1)}}} -\E{h\lr{\twiddle{\Sigma}_A^{(L+1)}}},
\]
where $h:\R^{n_{L+1}\times \abs{A}}\gives \R$ is a bounded continuous function and $\twiddle{\Sigma}_A^{(L+1)}$ is the conditional covariance matrix at layer $L+1$ for $\twiddle{z}_A^{(L+1)}$. Combining this with the convergence in distribution from \eqref{E:dist-conv} shows that \eqref{E:g-goal-2} holds and completes the proof of Proposition \ref{P:fdd} for general weight distributions. \hfill $\square$

\subsection{Tightness: Proof of Proposition \ref{P:tightness}}\label{S:tightness-proof}
In this section, we provide a proof of Proposition \ref{P:tightness}. In the course of showing tightness, we will need several elementary Lemmas, which we record in the \S \ref{S:prep-lem}. We then use them in \S \ref{S:tightness-proof-2} to complete the proof of Proposition \ref{P:tightness}.

\subsubsection{Preparatory Lemmas}\label{S:prep-lem}
For the first Lemma, let us agree to write $\mathcal C(A)$ for the cone over a subset $A$ in a euclidean space and $B_1(\R^n)$ for the unit ball in $\R^n$. 
\begin{lemma}\label{L:lip-image}
  Fix integers $n_0,n_1\geq 1$ and a real number $\lambda \geq 1$. Suppose that $T$ is a compact of $\R^{n_0}$ and $f:\R^{n_0}\gives \R^{n_1}$ is $\lambda$-Lipschitz with respect to the $\ell_2$-norm on both $\R^{n_0}$ and $\R^{n_1}$. Define the Minkowski sums
  \[
  \widehat{T}=f(T)+\mathcal C\lr{f(T)-f(T)}\cap B_1(\R^{n_1}).
  \]
  There exists a constant $C>0$ a compact subset $T'$ of $\R^{3n_0+1}$, and a $C\lambda$-Lipschitz map $g:\R^{3n_0+1}\gives \R^{n_1}$ (all depending only $T, \lambda$,), so that $\widehat{T} = g(T').$
\end{lemma}
\begin{proof}
By definition,
\begin{equation}\label{E:g-def}
\widehat{T} = g(T\times T\times T\times  [0,1]),\qquad g(x,y,z,r) = f(x)+r(f(y)-f(z)).    
\end{equation}
In particular, for some constant $C$ depending on $T_0, \lambda$, we have
\begin{align*}
    \norm{g(x,y,z,r)-g(x',y',z',r')}_2&\leq \norm{f(x)-f(x')}_2 + \norm{f(y)-f(y')}_2+\norm{f(z)-f(z')}_2\\
    &+ \abs{r-r'}\norm{f(y)-f(z)}_2\\
    &\leq \lambda\lr{\norm{x-x'}_2 + \norm{y-y'}_2 + \norm{z-z'}_2 + \mathrm{Diam}(T_0) \abs{r-r_0}}\\
    &\leq C \lambda \norm{(x-x',y-y',z-z',r-r')}_2.
\end{align*}
Hence, $\widehat{T}$ is the image under a Lipschitz map with a  Lipschitz constant depending only on $T_0,\lambda$ of a compact set in $\R^{3n_0+1}$.
\end{proof}

\noindent The second Lemma we need is an elementary inequality. 
\begin{lemma}\label{L:split-abc}
 Let $a,b,c\geq 0$ be real numbers and $k\geq 1$ be an integer. We have
 \[
(a+b+c)^k \leq 2^{2k-1} \lr{1 + a^{2k}}+\frac{1}{4}\left[(2+b)^{4k}+\lr{1+c}^{4k}\right].
\]
\end{lemma}
\begin{proof}
For any $a,b\geq 0$ we have
\begin{align}
\notag (a+b)^k &=\sum_{j=0}^{k}\binom{k}{j}a^jb^{k-j}\leq \sum_{j=0}^{k}\binom{k}{j}\lr{1+a}^kb^{k-j} = (1+a)^k(1+b)^k\\
\label{E:inq-1}&\leq \frac{1}{2}\lr{(1+a)^{2k} + (1+b)^{2k}}    
\end{align}
Further, breaking into cases depending on whether $0\leq a\leq 1$ or $1\leq a$ we find that
\begin{align}
    \label{E:inq-2}
    (a+b)^k \leq 2^{2k-1} \lr{1 + a^{2k}}+\frac{1}{2}(1+b)^{2k}.
\end{align}
Combining \eqref{E:inq-1} with \eqref{E:inq-2} we see as desired that any $a,b,c\geq 0$
\[
(a+b+c)^k \leq 2^{2k-1} \lr{1 + a^{2k}}+\frac{1}{4}\left[(2+b)^{4k}+\lr{1+c}^{4k}\right].
\]
\end{proof}

\noindent The next Lemma is also an elementary estimate. 

\begin{lemma}\label{L:CS-power}
 Fix an integer $k\geq 1$, and suppose $X_1,\ldots, X_k$ are non-negative random variables. There exists a positive integer $q=q(k) $ such that
 \[
 \E{\prod_{i=1}^k X_i} \leq \lr{\prod_{i=1}^k \E{X_i^q}}^{1/q}.
 \]
\end{lemma}
\begin{proof}
The proof is by induction on $k$. For the base cases when $k=1$, we may take $q=1$ and when $k=2$ we may take $q=2$ by Cauchy-Schwartz. Now suppose we have proved the claim for all $k=1,2,\ldots, K$ for some $K\geq 3$. Note that $1\leq v \lceil{(K+1)/2}\rceil\leq K$. So we may use Cauchy-Schwartz and the inductive hypothesis to obtain
\begin{align*}
   \E{\prod_{i=1}^{K+1} X_i}  &=  \E{\prod_{i=1}^{\lceil \frac{K+1}{2}\rceil } X_i \prod_{i=\lceil \frac{K+1}{2}\rceil+1}^K X_i}\\
   &\leq  \E{\prod_{i=1}^{\lceil \frac{K+1}{2}\rceil } X_i^2}^{1/2} \E{\prod_{i=\lceil \frac{K+1}{2}\rceil+1}^K X_i^2}^{1/2}\\
   &\leq \lr{\prod_{i=1}^{K+1} \E{X_i^{2q}}}^{1/2q},
\end{align*}
where $q=\max\set{q\lr{\lceil \frac{1}{2}(K+1)\rceil}, q\lr{K-\lceil \frac{1}{2}(K+1)\rceil-1}}$.
\end{proof}

\noindent The next Lemma is an elementary result about the moments of marginals of iid random vectors. 
\begin{lemma}\label{L:marginals}
 Fix an even integer $p\geq 2$ a positive integer, and suppose $\mu$ is a probability measure on $\R$  with mean $0$ and finite higher moments. Assume also that $w=\lr{w_1,\ldots, w_n}$ is a vector with iid components, each with distribution $\mu$. Then, there exists a constant $C$ depending only on $p$ and first $p$ moments of $\mu$ such that for all $n\geq 1$
\[
\sup_{\norm{u}=1} \E{\abs{w\cdot u}^{p}}\leq C.
\]
\end{lemma}
\begin{proof}
We will use the following result of \L ata\l a \cite[Thm. 2, Cor. 2, Rmk. 2]{latala1997estimation}. Suppose $X_i$ are independent random variables and $p$ is a positive even integer. Then 
\begin{equation}\label{E:latala}
\E{\abs{\sum_{i=1}^n X_i}^p} \simeq \inf\set{t>0~\bigg|~\sum_{i=1}^n \log \E{\abs{1+\frac{X_i}{t}}^p}\leq p},    
\end{equation}
where $\simeq$ means bounded above and below up to universal multiplicative constants. Let us fix a unit vector $u=\lr{u_1,\ldots, u_n}\in S^{n-1}$ and apply this to $X_i = u_i w_i$. Since $w_i$ have mean $0$ and $p$ is even we find
\begin{align*}
    \sum_{i=1}^n \log \E{\abs{1+\frac{X_i}{t}}^p} &\leq \sum_{i=1}^n \log\lr{1+ \sum_{k=2}^p \binom{p}{k} \frac{\abs{u_i}^k\E{\abs{w_1}^k}}{t^k}}.
\end{align*}
Note that for each $k=2,\ldots, p$ we have
\[
\E{\abs{w_1}^k} \leq \E{(1+\abs{w_1})^p}\qquad \text{and}\qquad \abs{u_i}^k \leq u_i^2.
\]
Hence, using that $\log(1+x)\leq x$ we find
\begin{align*}
    \sum_{i=1}^n \log \E{\abs{1+\frac{X_i}{t}}^p} &\leq \sum_{i=1}^n \log\lr{1+ u_i^2 \E{(1+\abs{w_1})^p}\sum_{k=2}^p \binom{p}{k} \frac{1}{t^k}}\\
    &= \sum_{i=1}^n \log\lr{1+ u_i^2 \E{(1+\abs{w_1})^p}\left[\lr{1+\frac{1}{t}}^p - 1 - \frac{p}{t}\right]}\\
    &\leq\E{(1+\abs{w_1})^p}\left[\lr{1+\frac{1}{t}}^p - 1 - \frac{p}{t}\right].
\end{align*}
Note that for $2<t$, there is a universal constant $C>0$ so that 
\[
\abs{\lr{1+\frac{1}{t}}^p - 1 - \frac{p}{t}} \leq \frac{Cp^2}{t^2}.
\]
Thus, there exists a constant $C'>0$ so that
\[
t> C'\sqrt{p\E{(1+\abs{w_1})^p})}\quad \Rightarrow \quad \sum_{i=1}^n \log \E{\abs{1+\frac{X_i}{t}}^p} \leq p.
\]
Combining this with \eqref{E:latala} completes the proof. 
\end{proof}

\noindent The final Lemma we need is an integrability statement for the  supremum of certain non-Gaussian fields over low-dimensional sets.

\begin{lemma}\label{L:moment-chaining}
Fix a positive integer $n_0$, an even integer $k\geq 1$, a compact set $T_0\subseteq \R^{n_0}$, a constant $\lambda >0$, and a probability measure $\mu$ on $\R$ with mean $0$, variance $1,$ and finite higher moments. For every $\epsilon \in (0,1)$ there exists a constant $C=C(T_0, \epsilon,\lambda, n_0, k,\mu)$ with the following property. Fix any integer $n_1\geq 1$ and a $\lambda$-Lipschitz map $f:\R^{n_0}\gives \R^{n_1}$. Define $T_1:=f(T_0)$, and let $w=\lr{w_1,\ldots, w_{n_1}}$ be a vector with iid components $w_i$, each drawn from $\mu$. Then, for any fixed $y_0\in T_1$
\begin{equation}\label{E:sup-moments}
\E{\sup_{y\in T_1} (w\cdot (y-y_0))^k}\leq C.    
\end{equation}
\end{lemma}
\begin{proof}
The proof is a standard chaining argument. For each $y\in T_1$ write $\Pi_k(y)$ for the closest point to $y$ in a $2^{-k}$ net in $T_1$ and assume without loss of generality that the diameter of $T_1$ is bounded above by $1$ and that $\Pi_0(y)=y_0$ for all $y\in T_1$. We have using the usual chaining trick that
\begin{align}
    \label{E:chaining-ub}\E{\lr{\sup_{y\in T_1} w\cdot (y-y_0)}^k}&\leq  \sum_{q_1,\ldots, q_k=0}^\infty \E{\prod_{i=1}^k \sup_{y_i\in T_1} \abs{w\cdot (\Pi_{q_i}(y_i)-\Pi_{q_i-1}(y_i))}}.
\end{align}
By Lemma \ref{L:CS-power}, there exists $q$ depending only on $k$ so that for any $q_1,\ldots, q_k$ we have
\begin{equation}\label{E:prod-est}
\E{\prod_{i=1}^k \sup_{y_i\in T_1} \abs{w\cdot (\Pi_{q_i}(y_i)-\Pi_{q_i-1}(y_i))}}\leq \prod_{i=1}^k \lr{\E{ \sup_{y\in T_1} \abs{w\cdot (\Pi_{q_i}(y)-\Pi_{q_i-1}(y))}^{q}}}^{1/q}.
\end{equation}
We seek to bound each expectation on the right hand side in \eqref{E:prod-est}. To do this, write
\[
\E{ \sup_{y\in T_1} \abs{w\cdot (\Pi_{q_i}(y)-\Pi_{q_i-1}(y))}^{q}} = \int_0^\infty \mathbb P\lr{\sup_{y\in T_1} \abs{w\cdot (\Pi_{q_i}(y)-\Pi_{q_i-1}(y))}^{q}> t}dt.
\]
Note that the supremum is only over a finite set of cardinality at most
\[
\abs{\Pi_{q_i}}\abs{\Pi_{q_{i-1}}}\leq 2^{cn_0q_i}
\]
for some $c>0$ depending only $T_0,\lambda$. This is because, by assumption $T_1$ is the image of $T_0$ under a $\lambda$-Lipschtiz map and Lipschitz maps preserve covering numbers. Thus, by a union bound, 
\[
 \mathbb P\lr{\sup_{y\in T_1} \abs{w\cdot (\Pi_{q_i}(y)-\Pi_{q_i-1}(y))}^{q}> t}\leq 2^{cn_0q_i} \sup_{y\in T_1}\mathbb P\lr{ \abs{w\cdot (\Pi_{q_i}(y)-\Pi_{q_i-1}(y))}^{q}> t}
\]
But for any $y\in T_1$ and any $s>0,p\geq 1$ we have
\begin{align*}
  \Prob{\abs{w\cdot (\Pi_q(y)-\Pi_{q-1}(y))}^{q} > s} &\leq \sup_{\norm{u}=1} \E{\abs{w\cdot u}^p} \lr{\frac{\norm{\Pi_k(y)-\Pi_{k-1}(y)}^p}{s^{p/q}}}\\
  &= 2^{-pq_i} s^{-p/q}\sup_{\norm{u}=1}\E{\abs{w\cdot u}^p}.
\end{align*}
Putting this all together we find for any $p\geq 2\max\set{q+2,cn_0}$ that
\[
\E{ \sup_{y\in T_1} \abs{w\cdot (\Pi_{q_i}(y)-\Pi_{q_i-1}(y))}^{q}} \leq 2^{-cn_0q_i} \sup_{\norm{u}=1}\E{\abs{w\cdot u}^p}.
\]
Thus, substituting this into \eqref{E:chaining-ub} yields
\begin{align*}
    \E{\lr{\sup_{y\in T_1} w\cdot (y-y_0)}^k}&\leq \sup_{\norm{u}=1}\E{\abs{w\cdot u}^p} \sum_{q_1,\ldots, q_k=0}^\infty 2^{-cn_0\sum_{i=1}^k q_i}\leq 2^k \sup_{\norm{u}=1}\E{\abs{w\cdot u}^p}.
\end{align*}
Appealing to Lemma \ref{L:marginals} completes the proof of Lemma \ref{L:moment-chaining}.
\end{proof}

\subsubsection{Proof of Proposition \ref{P:tightness} Using Lemmas from \S \ref{S:prep-lem}}\label{S:tightness-proof-2}

Let us first establish the equi-Lipschitz condition, which we recall states that for each $\epsilon\in (0,1)$ and each compact set $T\subseteq \R^{n_0}$ there exist $C>0$ so that with probability at least $1-\epsilon$ we have
\begin{equation}\label{E:equi-lip}
\sup_{x_\alpha,x_\beta\in T}\frac{ \norm{z_\alpha^{(L+1)}-z_\beta^{(L+1)}}_2}{\norm{x_\alpha-x_\beta}_2}\leq C.    
\end{equation}
For \eqref{E:equi-lip} to hold, 
we need a result about the Lipschitz constant of each layer. To ease the notation define a normalized single layer with random weights $W$ and random biases $b$  via the map $\psi:\R^{n_1}\gives\R^{n_2}$:
\begin{equation}\label{E:psi-def}
\psi(x;W,b) = \frac{1}{\sqrt{n_2}}\sigma\lr{Wx+b},    
\end{equation}
where $b\sim \mN(0,C_b\mathrm{I}_{n_2})$ and $W=\lr{w_{ij}}\in \R^{n_2\times n_1}$ with $w_{ij}$ drawn iid from a distribution with mean $0$, variance $1$, and finite higher moments. We choose the variance of $w_{ij}$ to be $1$ instead of $C_W/n_1$ since we will later think of $x$ as the normalized vector $(C_W/n_\ell)^{1/2}\sigma(z_\alpha^{(\ell)})$ of post-activations in a given layer.
\begin{lemma}\label{L:layer-lip}
Fix an integer $n_0\geq 1$, a compact set $T_0\subseteq \R^{n_0}$, and a constant $\lambda>0$. For every $\epsilon \in (0,1)$ there exists a constant $C=C(\epsilon, n_0, T_0, \sigma, \lambda)$ with the following property. Fix any integers $n_1,n_2\geq 1$, and define $\psi:\R^{n_1}\gives \R^{n_2}$ as in \eqref{E:psi-def}. Suppose that $T_1\subseteq \R^{n_1}$ is the image of $T_0$ under a $\lambda$-Lipschitz map from $\R^{n_0}$ to $\R^{n_1}$. Then,
\[
\mathbb P\lr{\sup_{x_\alpha,x_\beta\in T_1}\frac{\norm{\psi(x_\alpha)-\psi(x_\beta)}_2}{\norm{x_\alpha-x_\beta}_2} \leq C}\geq 1-\epsilon.
\]
\end{lemma}
\begin{proof}
Fix $x_\alpha\neq x_\beta\in T_1$ and define
\[
\xi_{\alpha\beta}=\frac{x_\alpha-x_\beta}{\norm{x_\alpha-x_\beta}_2}.
\]
Write $W_i$ for the $i$-th row of $W$  and $b_i$ for the $i$-th component of $b$. Since $ab\leq \frac{1}{2}(a^2+b^2)$ and  $\sigma$ is absolutely continuous, we have
\begin{align}
  \notag  \norm{\psi(x_\alpha)-\psi(x_\beta)}_2^2&\quad= \frac{1}{n_2}\sum_{i=1}^{n_2} \lr{\sigma\lr{W_i\cdot x_\alpha + b_i}-\sigma\lr{W_i\cdot x_\alpha + b_i}}^2\\
  \notag  &\quad=\frac{1}{n_2}\sum_{i=1}^{n_2} \lr{ W_i\cdot \xi_{\alpha\beta} \int_{0}^{\norm{x_\alpha-x_\beta}_2} \sigma'\lr{W_i\cdot \lr{x_\beta +t\xi_{\alpha\beta}} + b_i}dt}^2\\
\notag    &\quad \leq \norm{x_\alpha-x_\beta}_2^2 \frac{1}{n_2}\sum_{i=1}^{n_2} \sup_{y\in \widehat{T}}\lr{\sigma'\lr{W_i\cdot y + b_i}}^2 \sup_{\xi\in \twiddle{T}} \lr{W_i\cdot \xi}^2\\
 \label{E:diff-est-1}&\quad \leq \norm{x_\alpha-x_\beta}_2^2 \frac{1}{n_2}\sum_{i=1}^{n_2}\left[ \sup_{y\in \widehat{T}}\lr{\sigma'\lr{W_i\cdot y + b_i}}^4+ \sup_{\xi\in \twiddle{T}} \lr{W_i\cdot \xi}^4\right],
\end{align}
where we've set
\[
\twiddle{T} = \mathcal C(T_1-T_1)\cap S^{n_1-1}\qquad \text{and}\qquad 
\widehat{T} := T_1+\mathcal C(T_1-T_1)\cap B_1(\R^{n_1}),
\]
and have denoted by $\mathcal C(A)$ the cone over a set $A$ and by $B_1(\R^{n_1})$ the unit ball in $\R^{n_1}$. The estimate \eqref{E:diff-est-1} yields
\begin{align*}
    &\Prob{\sup_{x_\alpha,x_\beta\in T}\frac{\norm{\psi(x_\alpha)-\psi(x_\beta)}_2^2}{\norm{x_\alpha-x_\beta}_2^2} > C}\\
    &\qquad\leq \Prob{\frac{1}{n_2}\sum_{i=1}^{n_2}\left[ \sup_{y\in \widehat{T}}\lr{\sigma'\lr{W_i\cdot y + b_i}}^4+ \sup_{\xi\in \twiddle{T}} \lr{W_i\cdot \xi}^4\right] > C}
\end{align*}
Since $\sigma'$ is polynomially bounded by assumption \eqref{E:sigma-prop}, we find by Markov's inequality that there exists an even integer $k\geq 2$ so that for any $C>1$ 
\begin{align}
\Prob{\sup_{x_\alpha,x_\beta\in T}\frac{\norm{\psi(x_\alpha)-\psi(x_\beta)}^2}{\norm{x_\alpha-x_\beta}^2} > C}  
\label{E:lip-est-1}&\leq \frac{1+\E{\sup_{y\in \widehat{T}} \abs{W_1\cdot y + b_1}^k + \sup_{\xi\in \twiddle{T}} \abs{W_1\cdot \xi}^4}}{C-1}.
\end{align}
Our goal is now to show that the numerator in \eqref{E:lip-est-1} is bounded above by a constant that depends only on $T_0, n_0, \lambda$. For this, let us fix any $y_0\in \widehat{T}$ and apply Lemma \ref{L:split-abc} as follows:
\begin{align*}
\abs{W_1\cdot y + b_1}^k& = \abs{W_1\cdot (y-y_0) + W_1\cdot y_0 + b_1}^k\\
&\leq 2^{2k-1}\lr{1+\abs{W_1\cdot (y-y_0)}^{2k}}+\frac{1}{4}\left[(2+\abs{W_1\cdot y_0}^{4k})+\lr{1+\abs{b_1}}^{4k}\right]. 
\end{align*}
Substituting this and the analogous estimate for $\abs{W\cdot \xi}^4$ into \eqref{E:lip-est-1}, we see that since all moments of the entries of the weights and biases exist, there exists a constant $C'>0$ depending on $\lambda, T_0,k$ so that 
\begin{align}
\notag&\Prob{\sup_{x_\alpha,x_\beta\in T}\frac{\norm{\psi(x_\alpha)-\psi(x_\beta)}^2}{\norm{x_\alpha-x_\beta}^2} > C}\\  
\label{E:lip-est-2}&\qquad\leq \frac{C'+\E{\sup_{y\in \widehat{T}} \abs{W_1\cdot (y-y_0)}^{2k}+ \sup_{\xi\in \twiddle{T}} \abs{W_1\cdot (\xi-\xi_0)}^4}}{C-1},
\end{align}
where $\xi_0\in \twiddle{T}$ is any fixed point. Note that by Lemma \ref{L:lip-image}, the sets $\widehat{T},\twiddle{T}$ are both contained in the image of a compact subset $T'\subseteq \R^{3n_0+1}$ under a Lipschitz map, with Lipschitz constant depending only on $\lambda,  T$. Thus, an application of Lemma \ref{L:lip-image} shows that there exists a constant $C''$ depending only $\lambda,T,k$ so that
\[
\E{\sup_{y\in \widehat{T}} \abs{W_1\cdot (y-y_0)}^{2k}}+\E{\sup_{\xi\in \twiddle{T}} \abs{W_1\cdot(\xi-\xi_0)}^4}\leq C''.
\]
Substituting this into \eqref{E:lip-est-2} and taking $C$ sufficiently large completes the proof of Lemma \ref{L:lip-image}.
\end{proof}

Lemma \ref{L:layer-lip} directly yields the equi-Lipschitz estimate \eqref{E:equi-lip}. Indeed, let us fix $\epsilon \in (0,1)$ and a compact set $T\subseteq\R^{n_0}$. Let us define
\[
h_\alpha^{(\ell)}:=\begin{cases}
\sqrt{\frac{C_W}{n_0}} x_\alpha,&\qquad \ell=0\\
\sqrt{\frac{C_W}{n_\ell}} \sigma(z_\alpha^{(\ell)}),&\quad \ell=1,\ldots, L\\
\frac{1}{\sqrt{n_{L+1}}} z_{\alpha}^{(L+1)},&\qquad \ell = L+1
\end{cases}.
\]
For $\ell=1,\ldots, L+1$ we have
\[
h_\alpha^{(\ell)} =\begin{cases} \sqrt{\frac{C_W}{n_\ell}}\sigma\lr{\widehat{W}^{(\ell)}h_\alpha^{(\ell-1)}+b^{(\ell)}},&\qquad \ell=1,\ldots, L\\
\frac{1}{\sqrt{n_{L+1}}}\lr{\widehat{W}^{(L+1)}h_\alpha^{(L)}+b^{(L+1)}},&\qquad \ell = L+1
\end{cases}
\]
where the rescaled weight matrices $\widehat{W}^{(\ell)}$, defined in \eqref{E:W-def}, has iid mean $0$ variance $1$ entries with finite higher moments. To each of the transformations $h_\alpha^{(\ell)}\mapsto h_\alpha^{(\ell+1)}$ we may now apply Lemma \ref{L:layer-lip}. Specifically, applying Lemma \ref{L:layer-lip} in the first layer shows that there exists $C^{(1)}>0$ so that the rescaled first layer map
\[
\sup_{x_\alpha,x_\beta\in T}\frac{\norm{h_\alpha^{(1)}-h_\beta^{(1)}}}{\norm{x_\alpha-x_\beta}}\leq C^{(1)}
\]
with probability at least $1-\epsilon/(L+1)$. Thus, the image
\[
T^{(1)}:=h^{(1)}(T)
\]
of $T$ under the normalized first layer map is the image under a $C^{(1)}-$Lipschitz map of the compact set $T\subseteq \R^{n_0}$. This allows us to apply Lemma \ref{L:layer-lip} again, but this time to the second layer, to conclude that, again, there exist $C^{(2)}>0$ so that with probability at least $1-2\epsilon/(L+1)$ the normalized second layer map satisfies
\[
\sup_{x_\alpha,x_\beta\in T}\frac{\norm{h_\alpha^{(2)}-h_\beta^{(2)}}}{\norm{x_\alpha-x_\beta}}\leq\sup_{x_\alpha,x_\beta\in T}\frac{\norm{h_\alpha^{(2)}-h_\beta^{(2)}}}{\norm{h_\alpha^{(1)}-h_\beta^{(1)}}} \sup_{x_\alpha,x_\beta\in T}\frac{\norm{h_\alpha^{(1)}-h_\beta^{(1)}}}{\norm{x_\alpha-x_\beta}} \leq C^{(1)}C^{(2)}.
\]
Proceeding in this way, with probability at least $1-\epsilon$ we that
\[
\sup_{x_\alpha,x_\beta\in T}\frac{\norm{z_\alpha^{(L+1)}-z_\beta^{(L+1)}}}{\norm{x_\alpha-x_\beta}}=n_{L+1}^{1/2}\sup_{x_\alpha,x_\beta\in K}\frac{\norm{h_\alpha^{(L+1)}-h_\beta^{(L+1)}}}{\norm{x_\alpha-x_\beta}}\leq n_{L+1}^{1/2} C^{(1)}\cdots C^{(L+1)}.
\]
Since $n_{L+1}$ is fixed and finite, this confirms \eqref{E:equi-lip}. It remains to check the uniform boundedness condition in \eqref{E:aa-hyp}. For this note that for any fixed $x_\beta\in K$ by Lemma \ref{L:collective-properties}, we have 
\begin{align*}
\sup_{n_1,\ldots, n_L\geq 1}\E{\frac{1}{n_{L+1}}\norm{z_{\beta}^{(L+1)}}^2}<\infty.
\end{align*}
Thus, by Markov's inequality, $\norm{z_\beta^{(L+1)}}$ is bounded above with high probability. Combined with the equi-Lipschitz condition $\eqref{E:equi-lip}$, which we just saw holds with high probability on $K$, we conclude that for each $\epsilon>0$ there exists $C>0$ so that
\[
\mathbb P\lr{\sup_{x_\alpha\in K}\norm{z_\alpha^{(L+1)}}\leq C}\geq 1-\epsilon,
\]
as desired. \hfill $\square$

\bibliography{bibliography}{}

\begin{thebibliography}{10}

\bibitem{adlam2019random}
Ben Adlam, Jake Levinson, and Jeffrey Pennington.
\newblock A random matrix perspective on mixtures of nonlinearities for deep
  learning.
\newblock {\em arXiv preprint arXiv:1912.00827}, 2019.

\bibitem{adlam2020neural}
Ben Adlam and Jeffrey Pennington.
\newblock The neural tangent kernel in high dimensions: Triple descent and a
  multi-scale theory of generalization.
\newblock In {\em International Conference on Machine Learning}, pages 74--84.
  PMLR, 2020.

\bibitem{ahn2019fluctuations}
Andrew Ahn.
\newblock Fluctuations of beta-jacobi product processes.
\newblock {\em arXiv preprint arXiv:1910.00743}, 2019.

\bibitem{akemann2012universal}
Gernot Akemann and Zdzislaw Burda.
\newblock Universal microscopic correlation functions for products of
  independent ginibre matrices.
\newblock {\em Journal of Physics A: Mathematical and Theoretical},
  45(46):465201, 2012.

\bibitem{akemann2014universal}
Gernot Akemann, Zdzislaw Burda, and Mario Kieburg.
\newblock Universal distribution of lyapunov exponents for products of ginibre
  matrices.
\newblock {\em Journal of Physics A: Mathematical and Theoretical},
  47(39):395202, 2014.

\bibitem{akemann2019integrable}
Gernot Akemann, Zdzislaw Burda, and Mario Kieburg.
\newblock From integrable to chaotic systems: Universal local statistics of
  lyapunov exponents.
\newblock {\em EPL (Europhysics Letters)}, 126(4):40001, 2019.

\bibitem{bartlett2020benign}
Peter~L Bartlett, Philip~M Long, G{\'a}bor Lugosi, and Alexander Tsigler.
\newblock Benign overfitting in linear regression.
\newblock {\em Proceedings of the National Academy of Sciences},
  117(48):30063--30070, 2020.

\bibitem{belkin2019reconciling}
Mikhail Belkin, Daniel Hsu, Siyuan Ma, and Soumik Mandal.
\newblock Reconciling modern machine-learning practice and the classical
  bias--variance trade-off.
\newblock {\em Proceedings of the National Academy of Sciences},
  116(32):15849--15854, 2019.

\bibitem{brown2020language}
Tom~B. Brown, Benjamin Mann, Nick Ryder, Melanie Subbiah, Jared Kaplan,
  Prafulla Dhariwal, Arvind Neelakantan, Pranav Shyam, Girish Sastry, Amanda
  Askell, Sandhini Agarwal, Ariel Herbert-Voss, Gretchen Krueger, Tom Henighan,
  Rewon Child, Aditya Ramesh, Daniel~M. Ziegler, Jeffrey Wu, Clemens Winter,
  Christopher Hesse, Mark Chen, Eric Sigler, Mateusz Litwin, Scott Gray,
  Benjamin Chess, Jack Clark, Christopher Berner, Sam McCandlish, Alec Radford,
  Ilya Sutskever, and Dario Amodei.
\newblock Language models are few-shot learners.
\newblock 2020.

\bibitem{crisanti2012products}
Andrea Crisanti, Giovanni Paladin, and Angelo Vulpiani.
\newblock {\em Products of random matrices: in Statistical Physics}, volume
  104.
\newblock Springer Science \& Business Media, 2012.

\bibitem{daniely2016toward}
Amit Daniely, Roy Frostig, and Yoram Singer.
\newblock Toward deeper understanding of neural networks: The power of
  initialization and a dual view on expressivity.
\newblock {\em Proceedings of Neural Information Processing Systems}, 2016.

\bibitem{daubechies2021nonlinear}
Ingrid Daubechies, Ronald DeVore, Simon Foucart, Boris Hanin, and Guergana
  Petrova.
\newblock Nonlinear approximation and (deep) relu networks.
\newblock {\em Constructive Approximation}, pages 1--46, 2021.

\bibitem{devore2020neural}
Ronald DeVore, Boris Hanin, and Guergana Petrova.
\newblock Neural network approximation.
\newblock {\em Acta Numerica 2021 (to appear). arXiv:2012.14501}, 2020.

\bibitem{du2018gradient}
Simon~S Du, Xiyu Zhai, Barnabas Poczos, and Aarti Singh.
\newblock Gradient descent provably optimizes over-parameterized neural
  networks.
\newblock {\em Proceedings of International Conference on Representation
  Learning}, 2019.

\bibitem{eldan2021non}
Ronen Eldan, Dan Mikulincer, and Tselil Schramm.
\newblock Non-asymptotic approximations of neural networks by gaussian
  processes.
\newblock {\em arXiv preprint arXiv:2102.08668}, 2021.

\bibitem{fan2020spectra}
Zhou Fan and Zhichao Wang.
\newblock Spectra of the conjugate kernel and neural tangent kernel for
  linear-width neural networks.
\newblock {\em Proceedings of Neural Information Processing Systems}, 2020.

\bibitem{furstenberg1963noncommuting}
Harry Furstenberg.
\newblock Noncommuting random products.
\newblock {\em Transactions of the American Mathematical Society},
  108(3):377--428, 1963.

\bibitem{furstenberg1960products}
Harry Furstenberg and Harry Kesten.
\newblock Products of random matrices.
\newblock {\em The Annals of Mathematical Statistics}, 31(2):457--469, 1960.

\bibitem{garriga2018deep}
Adri{\`a} Garriga-Alonso, Carl~Edward Rasmussen, and Laurence Aitchison.
\newblock Deep convolutional networks as shallow gaussian processes.
\newblock {\em arXiv preprint arXiv:1808.05587}, 2018.

\bibitem{gorin2018gaussian}
Vadim Gorin and Yi~Sun.
\newblock Gaussian fluctuations for products of random matrices.
\newblock {\em To appear in American Journal of Mathematics. arXiv:1812.06532},
  2018.

\bibitem{hanin2018neural}
Boris Hanin.
\newblock Which neural net architectures give rise to exploding and vanishing
  gradients?
\newblock {\em Proceedings of Neural Information Processing Systems 2018},
  2018.

\bibitem{hanin2019universal}
Boris Hanin.
\newblock Universal function approximation by deep neural nets with bounded
  width and relu activations.
\newblock {\em Mathematics}, 7(10):992, 2019.

\bibitem{hanin2019finite}
Boris Hanin and Mihai Nica.
\newblock Finite depth and width corrections to the neural tangent kernel.
\newblock {\em Proceedings of Interntional Conference on Representation
  Learning}, 2019.

\bibitem{hanin2019products}
Boris Hanin and Mihai Nica.
\newblock Products of many large random matrices and gradients in deep neural
  networks.
\newblock {\em Communications in Mathematical Physics}, pages 1--36, 2019.

\bibitem{hanin2021non}
Boris Hanin and Grigoris Paouris.
\newblock Non-asymptotic results for singular values of gaussian matrix
  products.
\newblock {\em Geometric and Functional Analysis}, 31(2):268--324, 2021.

\bibitem{hanin2019complexity}
Boris Hanin and David Rolnick.
\newblock Complexity of linear regions in deep networks.
\newblock In {\em International Conference on Machine Learning}, pages
  2596--2604. PMLR, 2019.

\bibitem{hanin2019deep}
Boris Hanin and David Rolnick.
\newblock Deep relu networks have surprisingly few activation patterns.
\newblock {\em Proceedings Neural Information Processing Systems}, 2019.

\bibitem{hastie2019surprises}
Trevor Hastie, Andrea Montanari, Saharon Rosset, and Ryan~J Tibshirani.
\newblock Surprises in high-dimensional ridgeless least squares interpolation.
\newblock {\em arXiv preprint arXiv:1903.08560}, 2019.

\bibitem{hebb1949organization}
Donald~Olding Hebb.
\newblock The organization of behavior; a neuropsycholocigal theory.
\newblock {\em A Wiley Book in Clinical Psychology}, 62:78, 1949.

\bibitem{huang2020dynamics}
Jiaoyang Huang and Horng-Tzer Yau.
\newblock Dynamics of deep neural networks and neural tangent hierarchy.
\newblock In {\em International Conference on Machine Learning}, pages
  4542--4551. PMLR, 2020.

\bibitem{jacot2018neural}
Arthur Jacot, Franck Gabriel, and Cl{\'e}ment Hongler.
\newblock Neural tangent kernel: Convergence and generalization in neural
  networks.
\newblock {\em Proceedings of Neural Information Processing Systems}, 2018.

\bibitem{krizhevsky2012imagenet}
Alex Krizhevsky, Ilya Sutskever, and Geoffrey~E Hinton.
\newblock Imagenet classification with deep convolutional neural networks.
\newblock {\em Advances in neural information processing systems},
  25:1097--1105, 2012.

\bibitem{latala1997estimation}
Rafa{\l} Lata{\l}a.
\newblock Estimation of moments of sums of independent real random variables.
\newblock {\em The Annals of Probability}, 25(3):1502--1513, 1997.

\bibitem{lee2017deep}
Jaehoon Lee, Yasaman Bahri, Roman Novak, Samuel~S Schoenholz, Jeffrey
  Pennington, and Jascha Sohl-Dickstein.
\newblock Deep neural networks as gaussian processes.
\newblock {\em Proceedings of Interntional Conference on Representation
  Learning}, 2019.

\bibitem{liu2020linearity}
Chaoyue Liu, Libin Zhu, and Mikhail Belkin.
\newblock On the linearity of large non-linear models: when and why the tangent
  kernel is constant.
\newblock {\em Proceedings of Neural Information Processing Systems}, 2020.

\bibitem{matthews2018gaussian}
Alexander G de~G Matthews, Mark Rowland, Jiri Hron, Richard~E Turner, and
  Zoubin Ghahramani.
\newblock Gaussian process behaviour in wide deep neural networks.
\newblock {\em Proceedings of International Conference on Representation
  Learning}, 2018.

\bibitem{neal1996priors}
Radford~M Neal.
\newblock Priors for infinite networks.
\newblock In {\em Bayesian Learning for Neural Networks}, pages 29--53.
  Springer, 1996.

\bibitem{nica2006lectures}
Alexandru Nica and Roland Speicher.
\newblock {\em Lectures on the combinatorics of free probability}, volume~13.
\newblock Cambridge University Press, 2006.

\bibitem{noci2021precise}
Lorenzo Noci, Gregor Bachmann, Kevin Roth, Sebastian Nowozin, and Thomas
  Hofmann.
\newblock Precise characterization of the prior predictive distribution of deep
  relu networks.
\newblock {\em arXiv preprint arXiv:2106.06615}, 2021.

\bibitem{novak2018bayesian}
Roman Novak, Lechao Xiao, Jaehoon Lee, Yasaman Bahri, Greg Yang, Jiri Hron,
  Daniel~A Abolafia, Jeffrey Pennington, and Jascha Sohl-Dickstein.
\newblock Bayesian deep convolutional networks with many channels are gaussian
  processes.
\newblock {\em arXiv preprint arXiv:1810.05148}, 2018.

\bibitem{peche2019note}
S~P{\'e}ch{\'e}.
\newblock A note on the pennington-worah distribution.
\newblock {\em Electronic Communications in Probability}, 24:1--7, 2019.

\bibitem{pennington2018emergence}
Jeffrey Pennington, Samuel Schoenholz, and Surya Ganguli.
\newblock The emergence of spectral universality in deep networks.
\newblock In {\em International Conference on Artificial Intelligence and
  Statistics}, pages 1924--1932. PMLR, 2018.

\bibitem{pennington2019nonlinear}
Jeffrey Pennington and Pratik Worah.
\newblock Nonlinear random matrix theory for deep learning.
\newblock {\em Journal of Statistical Mechanics: Theory and Experiment},
  2019(12):124005, 2019.

\bibitem{roberts2021principles}
Daniel Roberts, Sho Yaida, and Boris Hanin.
\newblock The principles of deep learning theory.
\newblock {\em arXiv:2106.10165 (under contract at Cambridge University
  Press)}, 2021.

\bibitem{rosenblatt1958perceptron}
Frank Rosenblatt.
\newblock The perceptron: a probabilistic model for information storage and
  organization in the brain.
\newblock {\em Psychological review}, 65(6):386, 1958.

\bibitem{ruelle1979ergodic}
David Ruelle.
\newblock Ergodic theory of differentiable dynamical systems.
\newblock {\em Publications Math{\'e}matiques de l'Institut des Hautes
  {\'E}tudes Scientifiques}, 50(1):27--58, 1979.

\bibitem{silver2017mastering}
David Silver, Julian Schrittwieser, Karen Simonyan, Ioannis Antonoglou, Aja
  Huang, Arthur Guez, Thomas Hubert, Lucas Baker, Matthew Lai, Adrian Bolton,
  et~al.
\newblock Mastering the game of go without human knowledge.
\newblock {\em nature}, 550(7676):354--359, 2017.

\bibitem{voiculescu1986addition}
Dan Voiculescu.
\newblock Addition of certain non-commuting random variables.
\newblock {\em Journal of functional analysis}, 66(3):323--346, 1986.

\bibitem{wigner1958distribution}
Eugene~P Wigner.
\newblock On the distribution of the roots of certain symmetric matrices.
\newblock {\em Annals of Mathematics}, pages 325--327, 1958.

\bibitem{yaida2020non}
Sho Yaida.
\newblock Non-gaussian processes and neural networks at finite widths.
\newblock In {\em Mathematical and Scientific Machine Learning}, pages
  165--192. PMLR, 2020.

\bibitem{yang2019scaling}
Greg Yang.
\newblock Scaling limits of wide neural networks with weight sharing: Gaussian
  process behavior, gradient independence, and neural tangent kernel
  derivation.
\newblock {\em arXiv preprint arXiv:1902.04760}, 2019.

\bibitem{yang2019tensori}
Greg Yang.
\newblock Tensor programs i: Wide feedforward or recurrent neural networks of
  any architecture are gaussian processes.
\newblock {\em Proceedings of Neural Information Processing Systems}, 2019.

\bibitem{yang2020tensorii}
Greg Yang.
\newblock Tensor programs ii: Neural tangent kernel for any architecture.
\newblock {\em arXiv preprint arXiv:2006.14548}, 2020.

\bibitem{yang2020tensoriii}
Greg Yang.
\newblock Tensor programs iii: Neural matrix laws.
\newblock {\em arXiv preprint arXiv:2009.10685}, 2020.

\bibitem{yarotsky2017error}
Dmitry Yarotsky.
\newblock Error bounds for approximations with deep relu networks.
\newblock {\em Neural Networks}, 94:103--114, 2017.

\bibitem{yarotsky2018optimal}
Dmitry Yarotsky.
\newblock Optimal approximation of continuous functions by very deep relu
  networks.
\newblock In {\em Conference on Learning Theory}, pages 639--649. PMLR, 2018.

\bibitem{zavatone2021exact}
Jacob~A Zavatone-Veth and Cengiz Pehlevan.
\newblock Exact priors of finite neural networks.
\newblock {\em arXiv preprint arXiv:2104.11734}, 2021.

\bibitem{zhang2021understanding}
Chiyuan Zhang, Samy Bengio, Moritz Hardt, Benjamin Recht, and Oriol Vinyals.
\newblock Understanding deep learning (still) requires rethinking
  generalization.
\newblock {\em Communications of the ACM}, 64(3):107--115, 2021.

\end{thebibliography}
\bibliographystyle{plain}
\end{document}